\documentclass[11pt,a4paper]{article}
\usepackage[affil-it]{authblk}
\usepackage[utf8]{inputenc}
\usepackage{amsmath, amssymb, mathrsfs, amsthm, amsfonts}
\usepackage{latexsym}
\usepackage[pdftex]{hyperref}
\usepackage{a4wide}
\usepackage{color}
\usepackage{graphicx}

\newtheorem{theorem}{Theorem}[section]
\newtheorem{corollary}[theorem]{Corollary}
\newtheorem{lemma}[theorem]{Lemma}
\newtheorem{proposition}[theorem]{Proposition}

\theoremstyle{definition}
\newtheorem{definition}[theorem]{Definition}
\theoremstyle{remark}

\title{A Multiparametric Quon Algebra}

\author{Hery Randriamaro
\thanks{This research was supported by my mother \\
Lot II B 32 bis Faravohitra, 101 Antananarivo, Madagascar \\
e-mail: \texttt{hery.randriamaro@gmail.com}}}

\begin{document}

\maketitle

\begin{abstract}
\noindent The quon algebra is an approach to particle statistics introduced by Greenberg in order to provide a theory in which the Pauli exclusion principle and Bose statistics are violated by a small amount. We generalize these models by introducing a deformation of the quon algebra generated by a collection of operators $\mathtt{a}_i$, $i \in \mathbb{N}^*$ the set of positive integers, on an infinite dimensional module satisfying the $q_{i,j}$-mutator relations $\mathtt{a}_i \mathtt{a}_j^{\dag} - q_{i,j}\, \mathtt{a}_j^{\dag} \mathtt{a}_i = \delta_{i,j}$. The realizability of our model is proved by means of the Aguiar-Mahajan bilinear form on the chambers of hyperplane arrangements. We show that, for suitable values of $q_{i,j}$, the module generated by the particle states obtained by applying combinations of $\mathtt{a}_i$'s and $\mathtt{a}_i^{\dag}$'s to a vacuum state $|0\rangle$ is an indefinite Hilbert module. Furthermore, we refind the extended Zagier's conjecture established independently by Meljanac et al. and by Duchamp et al.

\bigskip 

\noindent \textsl{Keywords}: Quon Algebra, Indefinite Hilbert Module, Hyperplane Arrangement 

\smallskip

\noindent \textsl{MSC Number}: 05E15, 81R10 
\end{abstract}

\section{Introduction}

\noindent Denote by $\mathbb{C}[q_{i,j}]$ the polynomial ring $\mathbb{C}\big[q_{i,j}\ \big|\ i,j \in \mathbb{N}^*\big]$ with variables $q_{i,j}$. The quons are particles whose annihilation and creation operators obey the quon algebra which interpolates between fermions and bosons.

\begin{definition}
By \textbf{multiparametric quon algebra} is meant the free algebra $\mathbf{A}$, equal to $\mathbb{C}[q_{i,j}]\big[\mathtt{a}_i\ \big|\ i \in \mathbb{N}^*\big]$, and subject to the anti-involution $\dag$ exchanging $\mathtt{a}_i$ with $\mathtt{a}_i^{\dag}$ and to the commutation relations   
$$\mathtt{a}_i \mathtt{a}_j^{\dag} = q_{i,j}\, \mathtt{a}_j^{\dag} \mathtt{a}_i + \delta_{i,j},$$
where $\delta_{i,j}$ is the Kronecker delta.
\end{definition}

\noindent The multiparametric quon algebra is a generalization of the deformed quon algebra subject to the restriction $q_{i,j} = \bar{q}_{j,i}$ independently studied by Bozejko and Speicher \cite[§~3]{BoSp}, and by Meljanac and Svrtan \cite[§~1.1]{MeSv}. Their algebra is in turn a generalization of the deformed quon algebra investigated by Speicher subject to the restriction $q_{i,j} = q_{j,i}$ \cite{Sp}. Then, his algebra is a generalization of the quon algebra introduced by Greenberg \cite{Grb2} and studied by Zagier \cite[§~1]{Za} which is subject to the commutation relations $\mathtt{a}_i \mathtt{a}_j^{\dag} = q\, \mathtt{a}_j^{\dag} \mathtt{a}_i + \delta_{i,j}$ obeyed by the annihilation and creation operators of the quon particles, and generating a model of infinite statistics. Finally, the quon algebra is a generalization of the classical Bose and Fermi algebras corresponding to the restrictions $q=1$ and $q=-1$ respectively, as well as of the intermediate case $q=0$ suggested by Hegstrom and investigated by Greenberg \cite{Grb1}.

\smallskip

\noindent In a Fock-like representation, the generators of $\mathbf{A}$ are the linear operators $\mathtt{a}_i, \mathtt{a}_i^{\dag}: \mathbf{V} \rightarrow \mathbf{V}$ on an infinite dimensional $\mathbb{C}[q_{i,j}]$-vector module $\mathbf{V}$ satisfying the commutation relations
$$\mathtt{a}_i \mathtt{a}_j^{\dag} = q_{i,j}\, \mathtt{a}_j^{\dag} \mathtt{a}_i + \delta_{i,j},$$ 
and the relations $$\mathtt{a}_i|0\rangle = 0,$$
where $\mathtt{a}_i^{\dag}$ is the adjoint of $\mathtt{a}_i$, and $|0\rangle$ is a nonzero distinguished vector of $\mathbf{V}$. The $\mathtt{a}_i$'s are the annihilation operators and the $\mathtt{a}_i^{\dag}$'s the creation operators.

\smallskip

\noindent Define the $q_{i,j}$-conjugate $\tilde{P}$ of a monomial $\displaystyle P = \mu \prod_{u \in [n]} q_{i_u, j_u} \in \mathbb{C}[q_{i,j}]$, where $\mu \in \mathbb{C}$, by
$$\tilde{P} := \tilde{\mu} \prod_{s \in [n]} \tilde{q}_{i_s, j_s} \quad \text{with} \quad \tilde{\mu} = \bar{\mu},\ \text{and}\ \tilde{q}_{i_s, j_s} = q_{j_s, i_s},$$
and the $q_{i,j}$-conjugate of a monomial sum $Q = P_1 + \dots + P_k \in \mathbb{C}[q_{i,j}]$ by $\tilde{Q} = \tilde{P}_1 + \dots +\tilde{P}_k$.

\begin{definition}
An \textbf{indefinite inner product} on $\mathbf{V}$ is a map $(.,.): \mathbf{V} \times \mathbf{V} \rightarrow \mathbb{C}[q_{i,j}]$ such that, for $\mu \in \mathbb{C}$, and $\mathtt{u}, \mathtt{v}, \mathtt{w} \in \mathbf{V}$, we have
\begin{itemize}
\item $(\mu \mathtt{u}, \mathtt{v}) = \mu(\mathtt{u}, \mathtt{v})$ and $(\mathtt{u} + \mathtt{v}, \mathtt{w}) = (\mathtt{u}, \mathtt{w}) + (\mathtt{v}, \mathtt{w})$,
\item $(\mathtt{u}, \mathtt{v}) = \widetilde{(\mathtt{v}, \mathtt{u})}$,
\item and, if $\mathtt{u} \neq 0$, $(\mathtt{u}, \mathtt{u}) \in \mathbb{C}[q_{i,j}] \setminus \{0\}$.
\end{itemize}
\end{definition}

\noindent Let $\mathbf{H}$ be the submodule of $\mathbf{V}$ generated by the particle states obtained by applying combinations of $\mathtt{a}_i$'s and $\mathtt{a}_i^{\dag}$'s to $|0\rangle$, that is $\mathbf{H} := \big\{\mathtt{a} |0\rangle\ \big|\ \mathtt{a} \in \mathbf{A}\big\}$. The aim of this article is to prove the realizability of that model through the following theorem.

\begin{theorem} \label{ThMa}
Under the condition $|q_{i,j}| < 1$, the module $\mathbf{H}$ is an indefinite Hilbert module for the map $(.,.): \mathbf{H} \times \mathbf{H} \rightarrow \mathbb{C}[q_{i,j}]$ defined, for $\mu, \nu \in \mathbb{C}[q_{i,j}]$, and $\mathtt{a}, \mathtt{b} \in \mathbf{A}$, by
$$\big(\mu\mathtt{a}|0 \rangle,\, \nu\mathtt{b}|0 \rangle\big) := \mu \tilde{\nu} \, \langle 0| \mathtt{a} \, \mathtt{b}^{\dag} |0 \rangle  \quad \text{with} \quad \langle 0|0 \rangle = 1,$$
and where the usual bra-ket product $\langle 0| \mathtt{a} \, \mathtt{b}^{\dag} |0 \rangle$ is subject to the defining relations of $\mathbf{A}$.
\end{theorem}

\noindent The indefinite inner product of Theorem \ref{ThMa} becomes an inner product when the matrix representing $(.,.)$ is positively diagonalizable. Theorem \ref{ThMa} is particularly a generalization of the realizability of the deformed quon algebra model for $q_{i,j} = \bar{q}_{j,i}$ established independently by Bozejko and Speicher \cite[Corollary~3.2]{BoSp}, and by Meljanac and Svrtan \cite[Theorem~1.9.4]{MeSv}, which in turn is a generalization of the realizability of the deformed quon algebra model for $q_{i,j} = q_{j,i}$ established by Speicher \cite[Corollary]{Sp}, which finally is a generalization of the realizability of the quon algebra model established by Zagier \cite[Theorem~1]{Za}.

\noindent To prove Theorem \ref{ThMa}, we first show with Lemma \ref{H} that
$$\mathbf{B} := \big\{\mathtt{a}_{i_1}^{\dag} \dots \mathtt{a}_{i_n}^{\dag}|0\rangle\ \big|\ (i_1, \dots, i_n) \in (\mathbb{N}^*)^n,\, n \in \mathbb{N}\big\}$$ is a basis of $\mathbf{H}$, so that we can assume that $\displaystyle \mathbf{H} = \Big\{ \sum_{i=1}^n \mu_i \mathtt{b}_i\ \Big|\ n \in \mathbb{N}^*,\, \mu_i \in \mathbb{C}[q_{i,j}],\, \mathtt{b}_i \in \mathbf{B}\Big\}$.

\noindent The infinite matrix associated to the map of Theorem \ref{ThMa} is $\mathbf{M} := \big((\mathtt{b},\mathtt{a})\big)_{\mathtt{a},\mathtt{b} \in \mathbf{B}}$.

\noindent Let $\displaystyle \left(\!\!{\mathbb{N}^* \choose n}\!\!\right)$ be the set of multisets of $n$ elements in $\mathbb{N}^*$. We prove with Lemma \ref{zero} that
$$\mathbf{M} = \bigoplus_{n \in \mathbb{N}^*} \bigoplus_{\displaystyle I \in \left(\!\!{\mathbb{N}^* \choose n}\!\!\right)} \mathbf{M}_I \quad \text{with} \quad \mathbf{M}_I = \Big( \langle 0|\, \mathtt{a}_{\dot{\tau}(n)} \dots \mathtt{a}_{\dot{\tau}(1)} \, \mathtt{a}_{\dot{\sigma}(1)}^{\dag} \dots \mathtt{a}_{\dot{\sigma}(n)}^{\dag} \,|0 \rangle \Big)_{\dot{\sigma}, \dot{\tau} \in \mathfrak{S}_I},$$
where $\mathfrak{S}_I$ is the permutation set of the multiset $I$. For example,
$$\mathbf{M}_{[3]} = \begin{pmatrix}
1 & q_{3,2} & q_{2,1} & q_{2,1} q_{3,1} & q_{3,1} q_{3,2} & q_{3,1} q_{2,1} q_{3,2} \\
q_{2,3} & 1 & q_{2,1} q_{2,3} & q_{2,1} q_{3,1} q_{2,3} & q_{3,1} & q_{3,1} q_{2,1} \\
q_{1,2} & q_{1,2} q_{3,2} & 1 & q_{3,1} & q_{3,2} q_{1,2} q_{3,1} & q_{3,2} q_{3,1} \\
q_{1,2} q_{1,3} & q_{1,2} q_{1,3} q_{3,2} & q_{1,3} & 1 & q_{3,2} q_{1,2} & q_{3,2} \\
q_{1,3} q_{2,3} & q_{1,3} & q_{2,3} q_{2,1} q_{1,3} & q_{2,3} q_{2,1} & 1 & q_{2,1} \\
q_{1,3} q_{1,2} q_{2,3} & q_{1,3} q_{1,2} & q_{2,3} q_{1,3} & q_{2,3} & q_{1,2} & 1
\end{pmatrix}.$$

\noindent Proposition \ref{Var} and Lemma \ref{Ch} permits us to deduce that, if $\displaystyle J \in \binom{\mathbb{N}^*}{n} \subseteq \left(\!\!{\mathbb{N}^* \choose n}\!\!\right)$, then
$$\det \mathbf{M}_J = \prod_{\substack{K \in 2^{J} \\ \#K \, \geq 2}} \Big(1- \prod_{\{s,t\} \in \binom{K}{2}} q_{s,t} q_{t,s}\Big)^{(\#K\,- 2)!\,(n- \, \#K \, +1)!}.$$
For example, $\det \mathbf{M}_{[3]} = (1 - q_{1,2} q_{2,1})^2 \, (1 - q_{1,3} q_{3,1})^2 \, (1 - q_{2,3} q_{3,2})^2 \, (1 - q_{1,2} q_{2,1} q_{1,3} q_{3,1} q_{2,3} q_{3,2})$.

\noindent That determinant was independently computed by Meljanac and Svrtan for the specialization $q_{i,j} = \bar{q}_{j,i}$ \cite[Theorem~1.9.2]{MeSv}, by Duchamp et al. for the specialization $q_{i,j} = q_{j,i}$ \cite[§ 6.4.1]{DuEtAl}, and by Zagier for the specialization $q_{i,j} = q_{j,i} = q$ \cite[Theorem~2]{Za}.

\noindent Moreover, consider the multiset $\displaystyle I = \{\overbrace{i_1, \dots, i_1}^{\text{$p_1$ times}}, \overbrace{i_2, \dots, i_2}^{\text{$p_2$ times}}, \dots, \overbrace{k, \dots, k}^{\text{$p_k$ times}}\} \in \left(\!\!{\mathbb{N}^* \choose n}\!\!\right)$. For $s \in [n]$, let $\dot{s} := i_j$ if $s \in [p_j + p_{j-1} + \dots + p_1] \setminus [p_{j-1} + \dots + p_1]$. Suppose that the generators of $\mathbf{A}$ satisfy the commutation relations $\mathtt{a}_s \mathtt{a}_t^{\dag} = q_{\dot{s},\dot{t}}\, \mathtt{a}_t^{\dag} \mathtt{a}_s + \delta_{s,t}$. In that case, if we regard $\mathbf{M}_{[n]}$ as the matrix representing a linear map $\alpha: M \rightarrow M$ on a module $M$, then we prove with Proposition~\ref{Rep} and Lemma~\ref{CoSym} that $\mathbf{M}_I$ is the matrix representing $\alpha$ restricted to a submodule $N \subseteq M$ such that $\alpha(N) = N$. Therefore, we can infer that, for every $\displaystyle I \in \left(\!\!{\mathbb{N}^* \choose n}\!\!\right)$, $\mathbf{M}_I$ is nonsingular for $|q_{i,j}| < 1$, $i,j \in \mathbb{N}^*$.

\smallskip

\noindent When, for special values of the $q_{i,j}$'s, $\mathbf{M}_{[n]}$ is diagonalizable, then $\mathbf{M}_I$ becomes positive definite. Indeed, as $\mathbf{M}_I$ is the identity matrix if $q_{i,j}=0$, for every $i,j \in \mathbb{N}^*$, we deduce by continuity that $\mathbf{M}_I$ is positive definite. For these suitable values of $q_{i,j}$, $\mathbf{M}$ becomes a positive definite matrix or, in other terms, the map in Theorem~\ref{ThMa} becomes an inner product on $\mathbf{H}$. It is the case of the algebras investigated by Meljanac and Svrtan, and Zagier since, with their models, $\mathbf{M}_{[n]}$ is a hermitian matrix, that is consequently diagonalizable.

\smallskip

\noindent Finally, we provide another proof of the extended Zagier's conjecture in Section \ref{SecZ}.

\begin{proposition} \label{CoZa}
Let $n \in \mathbb{N}^*$, and assume that the generators of $\mathbf{A}$ satisfy the commutation relations $\mathtt{a}_s \mathtt{a}_t^{\dag} = q\, \mathtt{a}_t^{\dag} \mathtt{a}_s + \delta_{s,t}$. Then, each entry of $\mathbf{M}_{[n]}^{-1}$ is an element of $\displaystyle \frac{\mathbb{C}[q]}{\displaystyle \prod_{i \in [n-1]}\big(1 - q^{i^2 + i}\big)^{n-i}}$.
\end{proposition}

\noindent The extended Zagier's conjecture was first established by Meljanac and Svrtan \cite[Corollary~2.2.8]{MeSv} who disproved Zagier's conjecture by using the algorithm of \cite[Proposition~2.2.15]{MeSv} for $n=8$. One deduces also Proposition \ref{CoZa} from the study of the representation of the permutation group made by Duchamp et al. \cite[Proposition~4.6, 4.9]{DuEtAl}.

\section{Hyperplane Arrangements}  \label{SeHy}

\noindent We establish two results we need concerning the hyperplane arrangement associated to the permutation group of $n$ elements in order to prove Theorem \ref{ThMa}.

\smallskip

\noindent Recall that a \textbf{hyperplane} in the space $\mathbb{R}^n$ is a $(n-1)$-dimensional linear subspace, and a \textbf{hyperplane arrangement} is a finite set of hyperplanes. For a hyperplane $H$, denote its two associated open half-spaces by $H^+$ and $H^-$, and let $H^0 := H$. A \textbf{face} of a hyperplane arrangement $\mathcal{A}$ is a subset of $\mathbb{R}^n$ having the form $$F := \bigcap_{H \in \mathcal{A}} H^{\epsilon_H(F)} \quad \text{with} \quad \epsilon_H(F) \in \{+,0,-\}.$$ 

\noindent A \textbf{chamber} of $\mathcal{A}$ is a face $C \in F_{\mathcal{A}}$ such that, for every $H \in \mathcal{A}$, $\epsilon_H(F) \neq 0$. Denote the set formed by the chambers of $\mathcal{A}$ by $C_{\mathcal{A}}$. Assign a variable $h_H^{\varepsilon}$, $\varepsilon \in \{+,-\}$, to every half-space $H^{\varepsilon}$. We work with the polynomial ring $R_{\mathcal{A}} := \mathbb{Z}\big[h_H^{\varepsilon}\ \big|\ H \in \mathcal{A},\, \varepsilon \in \{+,-\}\big]$, and the module of $R_{\mathcal{A}}$-linear combinations of chambers $\displaystyle M_{\mathcal{A}} := \Big\{\sum_{C \in C_{\mathcal{A}}} x_C C\ \Big|\ x_C \in R_{\mathcal{A}}\Big\}$.

\noindent For $C,D \in C_{\mathcal{A}}$, let $H_{C,D}$ be the set of half-spaces $\big\{H^{\epsilon_H(C)}\ \big|\ H \in \mathcal{A},\, \epsilon_H(C) = - \epsilon_H(D)\big\}$.

\noindent The \textbf{Aguiar-Mahajan bilinear form} $\mathrm{v}: M_{\mathcal{A}} \times M_{\mathcal{A}} \rightarrow R_{\mathcal{A}}$ is defined, for $C,D \in C_{\mathcal{A}}$, by 
$$\mathrm{v}(C,C) = 1 \quad \text{and} \quad \mathrm{v}(C,D) = \prod_{H^{\varepsilon} \in H_{C,D}} h_H^{\varepsilon}\,\ \text{if}\,\ C \neq D.$$

\noindent From $\mathrm{v}$ is defined the linear map $\gamma: M_{\mathcal{A}} \rightarrow M_{\mathcal{A}}$, for $D \in C_{\mathcal{A}}$, by $\displaystyle \gamma(D) := \sum_{C \in C_{\mathcal{A}}} \mathrm{v}(D,C)\, C.$ 

\noindent Let $x = (x_1, \dots, x_n)$ be a variable of $\mathbb{R}^n$. We precisely investigate the hyperplane arrangement associated to the permutation group $\mathfrak{S}_n$ of $n$ elements defined by
$$\mathcal{A}_n := \{H_{i,j}\ |\ i,j \in [n],\, i<j\} \quad \text{with} \quad H_{i,j} = \{x \in \mathbb{R}^n\ |\ x_i = x_j\}.$$

\noindent The set formed by the chambers of $\mathcal{A}_n$ is
$$C_{\mathcal{A}_n} = \{C_{\sigma}\ |\ \sigma \in \mathfrak{S}_n\} \quad \text{with} \quad C_{\sigma} := \{x \in \mathbb{R}^n\ |\ x_{\sigma(1)} < x_{\sigma(2)} < \dots < x_{\sigma(n)}\}.$$

\noindent For $i,j \in [n]$ with $i \neq j$, assign the variable $q_{i,j}$ to the half-space $\{x \in \mathbb{R}^n\ |\ x_i < x_j\}$. On $\mathcal{A}_n$, the ring $R_{\mathcal{A}_n} := \mathbb{Z}\big[q_{i,j}\ \big|\ i,j \in [n]\big]$ and the module $\displaystyle M_{\mathcal{A}_n} := \Big\{\sum_{\sigma \in \mathfrak{S}_n} x_{\sigma} C_{\sigma}\ \Big|\ x_{\sigma} \in R_{\mathcal{A}_n}\Big\}$ are considered. The Aguiar-Mahajan bilinear form becomes $\mathrm{v}_n: M_{\mathcal{A}_n} \times M_{\mathcal{A}_n} \rightarrow R_{\mathcal{A}_n}$ defined, for $C_{\sigma},C_{\tau} \in C_{\mathcal{A}_n}$, by
$$\mathrm{v}_n(C_{\sigma},C_{\tau}) = \prod_{\substack{\{i,j\} \, \in \, \binom{[n]}{2} \\ i \, < \, j \\ \tau^{-1} \circ \sigma(i) \, > \, \tau^{-1} \circ \sigma(j)}} q_{\sigma(i),\sigma(j)},$$
and $\gamma$ the linear map $\gamma_n: M_{\mathcal{A}_n} \rightarrow M_{\mathcal{A}_n}$ defined, for $C_{\tau} \in C_{\mathcal{A}_n}$, by
$$\gamma_n(C_{\tau}) := \sum_{\sigma \in \mathfrak{S}_n} \mathrm{v}_n(C_{\tau},C_{\sigma})\, C_{\sigma}.$$

\begin{proposition} \label{Var}
For an integer $n \geq 2$, we have $$\det \gamma_n = \prod_{\substack{I \in 2^{[n]} \\ \#I \, \geq 2}} \Big(1- \prod_{\{i,j\} \in \binom{I}{2}} q_{i,j} q_{j,i}\Big)^{(\#I\,- 2)!\,(n- \, \#I \, +1)!}.$$
\end{proposition}

\begin{proof}
We first discuss about the general case of hyperplane arrangements. A flat of $\mathcal{A}$ is an intersection of hyperplanes in $\mathcal{A}$. Denote the set formed by the flats of $\mathcal{A}$ by $L_{\mathcal{A}}$. The weight of a flat $E \in L_{\mathcal{A}}$ is the monomial $\displaystyle \mathrm{b}_E := \prod_{\substack{H \in \mathcal{A} \\ E \subseteq H}} h_H^+ h_H^-$, and, if we choose a hyperplane $H$ containing $E$, the multiplicity $\beta_E$ of $E$ is half the number of chambers $C \in C_{\mathcal{A}}$ which have the property that $E$ is the minimal edge containing $\overline{C} \cap H$. Aguiar and Mahajan proved that \cite[Theorem~8.11]{AgMa} $$\det \gamma = \prod_{E \in L_{\mathcal{A}}} (1 - \mathrm{b}_E)^{\beta_E}.$$ 
Now, concerning $\mathcal{A}_n$, let $L_{\mathcal{A}_n}' = \{E \in L_{\mathcal{A}_n}\ |\ \beta_E \neq 0\}$. For a subset $I \subseteq [n]$ with $\#I \geq 2$, denote by $E_I$ the edge $\displaystyle \bigcap_{\{i,j\} \in \binom{I}{2}} H_{i,j}$. Randriamaro proved that \cite[Lemma 3.2, 3.3]{Ra} $$L_{\mathcal{A}_n}' = \big\{E_I\ \big|\ I \subseteq [n],\, |I| \geq 2\big\} \quad \text{and} \quad \beta_{E_I} = (\#I\,-2)!\,(n-\, \#I \,+1)!.$$
\end{proof}

\noindent Take a partition $\lambda = (p_1, \dots, p_k) \in \mathrm{Par}(n)$ of $n$. Denote by $\mathfrak{S}_{\lambda}$ the subgroup $\displaystyle \prod_{i \in [k]} \mathfrak{S}_{\lambda_i}$ of $\mathfrak{S}_n$, where $\mathfrak{S}_{\lambda_i}$ is the permutation group of the set $[p_i + p_{i-1} + \dots + p_1] \setminus [p_{i-1} + \dots + p_1]$.

\noindent Consider the multiset $I_{\lambda} = \{\overbrace{1, \dots, 1}^{\text{$p_1$ times}}, \overbrace{2, \dots, 2}^{\text{$p_2$ times}}, \dots, \overbrace{k, \dots, k}^{\text{$p_k$ times}}\}$. Denote by $\mathfrak{S}_{I_{\lambda}}$ the permutation set of the multiset $I_{\lambda}$. For $s \in [n]$, define $\dot{s} := i$ if $s \in [p_i + p_{i-1} + \dots + p_1] \setminus [p_{i-1} + \dots + p_1]$.

\noindent Let $\mathrm{p}: \mathfrak{S}_n \rightarrow \mathfrak{S}_{I_{\lambda}}$ be the projection
$\displaystyle \mathrm{p}(\sigma) := \dot{\sigma(1)}\, \dot{\sigma(2)} \, \dots \, \dot{\sigma(n)}$. For $\dot{\sigma} \in \mathfrak{S}_{I_{\lambda}}$, define the element $\displaystyle C_{\dot{\sigma}} := \sum_{\sigma \in \mathrm{p}^{-1}(\dot{\sigma})} C_{\sigma} \in M_{\mathcal{A}_n}$. Denote by $M_{\mathcal{A}_n}^{\lambda}$ the submodule $$M_{\mathcal{A}_n}^{\lambda} := \Big\{\sum_{\dot{\sigma} \in \mathfrak{S}_{I_{\lambda}}} x_{C_{\dot{\sigma}}} C_{\dot{\sigma}}\ \Big|\ x_{C_{\dot{\sigma}}} \in R_{\mathcal{A}_n}\Big\}.$$

\noindent For $s,t \in [n]$ with $s \neq t$, assign the variable $q_{\dot{s},\dot{t}}$ to the half-space $\{x \in \mathbb{R}^n\ |\ x_s < x_t\}$.

\begin{proposition}  \label{Rep}
Let $n \in \mathbb{N}^*$, and $\lambda \in \mathrm{Par}(n)$. Then, $\gamma_n(M_{\mathcal{A}_n}^{\lambda}) = M_{\mathcal{A}_n}^{\lambda}$.
\end{proposition}

\begin{proof}
If $\sigma \in \mathfrak{S}_n$ such that $\mathrm{p}(\sigma) = \dot{\sigma} \in \mathfrak{S}_{I_{\lambda}}$, then $\mathrm{p}^{-1}(\dot{\sigma}) = \mathfrak{S}_{\lambda} \sigma$. Let $\upsilon \sigma \in \mathfrak{S}_{\lambda} \sigma$, and $\dot{\tau} \in \mathfrak{S}_{I_{\lambda}}$. If $\mathrm{p}(\tau) = \dot{\tau}$,
\begin{align*}
\mathrm{v}_n(C_{\dot{\tau}}, C_{\upsilon \sigma}) & = \sum_{\varphi \in \mathfrak{S}_{\lambda}} \mathrm{v}_n(C_{\varphi \tau}, C_{\upsilon \sigma}) \\
& = \sum_{\varphi \in \mathfrak{S}_{\lambda}} \prod_{\substack{\{i,j\} \, \in \, \binom{[n]}{2} \\ i \, < \, j \\ \sigma^{-1} \upsilon^{-1} \circ \varphi \tau(i) \, > \, \sigma^{-1} \upsilon^{-1} \circ \varphi \tau(j)}} q_{\dot{\varphi \tau(i)}, \dot{\varphi \tau(j)}} \\
& = \sum_{\varphi \in \mathfrak{S}_{\lambda}} \prod_{\substack{\{i,j\} \, \in \, \binom{[n]}{2} \\ i \, < \, j \\ \sigma^{-1} \upsilon^{-1} \circ \varphi \tau(i) \, > \, \sigma^{-1} \upsilon^{-1} \circ \varphi \tau(j)}} q_{\dot{\tau(i)}, \dot{\tau(j)}} \\
& = \sum_{\varphi \in \mathfrak{S}_{\lambda}} \prod_{\substack{\{i,j\} \, \in \, \binom{[n]}{2} \\ i \, < \, j \\ \sigma^{-1} \circ \varphi \tau(i) \, > \, \sigma^{-1} \circ \varphi \tau(j)}} q_{\dot{\tau(i)}, \dot{\tau(j)}} \\
& = \mathrm{v}_n(C_{\dot{\tau}}, C_{\sigma}).
\end{align*}
Hence, 
\begin{align*}
\gamma_n(C_{\dot{\tau}}) & = \sum_{\sigma \in \mathfrak{S}_n} \mathrm{v}_n(C_{\dot{\tau}}, C_{\sigma}) \, C_{\sigma} \\
& = \sum_{\dot{\sigma} \in \mathfrak{S}_{I_{\lambda}}} \sum_{\upsilon \in \mathfrak{S}_{\lambda}} \mathrm{v}_n(C_{\dot{\tau}}, C_{\upsilon \sigma}) \, C_{\upsilon \sigma} \\ 
& = \sum_{\dot{\sigma} \in \mathfrak{S}_{I_{\lambda}}} \sum_{\upsilon \in \mathfrak{S}_{\lambda}} \mathrm{v}_n(C_{\dot{\tau}}, C_{\sigma}) \, C_{\upsilon \sigma} \\
& = \sum_{\dot{\sigma} \in \mathfrak{S}_{I_{\lambda}}} \mathrm{v}_n(C_{\dot{\tau}}, C_{\sigma}) \sum_{\upsilon \in \mathfrak{S}_{\lambda}} C_{\upsilon \sigma} \\
& = \sum_{\dot{\sigma} \in \mathfrak{S}_{I_{\lambda}}} \mathrm{v}_n(C_{\dot{\tau}}, C_{\sigma}) \, C_{\dot{\sigma}}.
\end{align*} 
\end{proof}

\section{The Bra-Ket Product on $\mathbf{H}$}

\noindent We prove some useful properties of the map in Theorem~\ref{ThMa}. We particularly connect it to the bilinear form $\mathrm{v}_n$ defined in Section \ref{SeHy}.

\begin{lemma}  \label{H}
The vector space generated by our particle states is
$$\mathbf{H} = \Big\{ \sum_{i=1}^n \mu_i \mathtt{b}_i\ \Big|\ n \in \mathbb{N}^*,\, \mu_i \in \mathbb{C}[q_{i,j}],\, \mathtt{b}_i \in \mathbf{B}\Big\}.$$
\end{lemma}

\begin{proof}
Let $i \in \mathbb{N}^*$. We have,
\begin{align*}
\mathtt{a}_i\, \mathtt{a}_{j_1}^{\dag} \dots \mathtt{a}_{j_t}^{\dag} =\ & q_{i,j_1} \dots q_{i,j_t} \, \mathtt{a}_{j_1}^{\dag} \dots \mathtt{a}_{j_t}^{\dag}\, \mathtt{a}_i \\
& + \sum_{\substack{u \in [t] \\ j_u=i}} q_{i,j_1} \dots q_{i,j_{u-1}} \, \mathtt{a}_{j_1}^{\dag} \dots \widehat{\mathtt{a}_{j_u}^{\dag}} \dots \mathtt{a}_{j_t}^{\dag},
\end{align*}
where the hat over the $u^{\text{th}}$ term of the product indicates that this term is omitted. So
$$\mathtt{a}_i\, \mathtt{a}_{j_1}^{\dag} \dots \mathtt{a}_{j_t}^{\dag}\, |0\rangle  = \sum_{\substack{u \in [t] \\ j_u=i}} q_{i,j_1} \dots q_{i,j_{u-1}} \, \mathtt{a}_{j_1}^{\dag} \dots \widehat{\mathtt{a}_{j_u}^{\dag}} \dots \mathtt{a}_{j_t}^{\dag} |0\rangle.$$
Thus, one can recursively remove every annihilation operator $\mathtt{a}_i$ of an element $\mathtt{a}|0\rangle \in \mathbf{H}$.
\end{proof}

\begin{lemma} \label{zero}
Let $(i_1, \dots, i_s) \in (\mathbb{N}^*)^s$ and $(j_1, \dots, j_t) \in (\mathbb{N}^*)^t$. If, as multisets, $\{i_1, \dots, i_s\}$ is different from $\{j_1, \dots, j_t\}$, then $\langle 0|\, \mathtt{a}_{i_s} \dots \mathtt{a}_{i_1} \, \mathtt{a}_{j_1}^{\dag} \dots \mathtt{a}_{j_t}^{\dag} \,|0 \rangle = 0$. 
\end{lemma}

\begin{proof}
Suppose that $v$ is the smallest integer in $[s]$ such that $i_v \notin \{j_1, \dots, j_t\} \setminus \{i_1, \dots, i_{v-1}\}.$ Then
$$\mathtt{a}_{i_s} \dots \mathtt{a}_{i_1} \, \mathtt{a}_{j_1}^{\dag} \dots \mathtt{a}_{j_t}^{\dag} = P\, \mathtt{a}_{i_v} \dots \mathtt{a}_{i_1} + Q\, \mathtt{a}_{i_v}\quad \text{with}\quad P,Q \in \mathbf{A}.$$
We deduce that $\mathtt{a}_{i_s} \dots \mathtt{a}_{i_1} \, \mathtt{a}_{j_1}^{\dag} \dots \mathtt{a}_{j_t}^{\dag}\, |0\rangle = P\, \mathtt{a}_{i_v} \dots \mathtt{a}_{i_1}\, |0\rangle + Q \, \mathtt{a}_{i_v}\, |0\rangle = 0$.\\
Similarly, suppose that $u$ is the smallest integer in $[t]$ such that $j_u$ does not belong to the multiset $\{i_1, \dots, i_s\} \setminus \{j_1, \dots, j_{u-1}\}.$ Then
$$\mathtt{a}_{i_s} \dots \mathtt{a}_{i_1} \, \mathtt{a}_{j_1}^{\dag} \dots \mathtt{a}_{j_t}^{\dag} = \mathtt{a}_{j_1}^{\dag} \dots \mathtt{a}_{j_u}^{\dag} \, P' + \mathtt{a}_{j_u}^{\dag}\, Q' \quad \text{with} \quad P',Q' \in \mathbf{A}.$$
And $\langle 0|\, \mathtt{a}_{i_s} \dots \mathtt{a}_{i_1} \, \mathtt{a}_{j_1}^{\dag} \dots \mathtt{a}_{j_t}^{\dag} = \langle 0|\, \mathtt{a}_{j_1}^{\dag} \dots \mathtt{a}_{j_u}^{\dag}\, P' +\langle 0|\, \mathtt{a}_{j_u}^{\dag}\, Q' = 0$.
\end{proof}

\noindent Therefore, we just need to investigate the product $\langle 0|\, \mathtt{a}_{i_n} \dots \mathtt{a}_{i_1} \, \mathtt{a}_{j_1}^{\dag} \dots \mathtt{a}_{j_n}^{\dag} \,|0 \rangle$ where $(j_1, \dots, j_n)$ is a multiset permutation of $(i_1, \dots, i_n)$.

\begin{lemma}  \label{Ch}
Let $\sigma, \tau \in \mathfrak{S}_n$, and $C_{\sigma}, C_{\tau} \in C_{\mathcal{A}_n}$ their associated chambers. Then,
$$\langle 0|\, \mathtt{a}_{\sigma(n)} \dots \mathtt{a}_{\sigma(1)} \, \mathtt{a}_{\tau(1)}^{\dag} \dots \mathtt{a}_{\tau(n)}^{\dag} \,|0 \rangle = \mathrm{v}_n(C_{\sigma}, C_{\tau}).$$
\end{lemma}

\begin{proof}
We have
\begin{align*}
\langle 0|\, \mathtt{a}_{\sigma(n)} \dots \mathtt{a}_{\sigma(1)} \, \mathtt{a}_{\tau(1)}^{\dag} \dots \mathtt{a}_{\tau(n)}^{\dag} \,|0 \rangle & = \prod_{s \in [n]}\ \prod_{\substack{t \in [n] \setminus [s] \\ \tau^{-1} \circ \sigma(s) \, > \, \tau^{-1} \circ \sigma(t)}}q_{\sigma(s),\sigma(t)} \\
& = \prod_{\substack{\{s,t\} \, \in \, \binom{[n]}{2} \\ s \, < \, t \\ \tau^{-1} \circ \sigma(s) \, > \, \tau^{-1} \circ \sigma(t)}} q_{\sigma(s),\sigma(t)} \\
& = \mathrm{v}_n(C_{\sigma}, C_{\tau}).
\end{align*}
\end{proof}

\noindent For $s,t \in [n]$ with $s \neq t$, assign the variable $q_{\dot{s},\dot{t}}$ to the half-space $\{x \in \mathbb{R}^n\ |\ x_s < x_t\}$.

\begin{lemma} \label{CoSym}
Let $\lambda \in \mathrm{Par(n)}$, and $\dot{\sigma}, \dot{\tau} \in \mathfrak{S}_{I_{\lambda}}$. Then, for every $\tau \in \mathrm{p}^{-1}(\dot{\tau})$, $$\langle 0|\, \mathtt{a}_{\dot{\sigma}(n)} \dots \mathtt{a}_{\dot{\sigma}(1)} \, \mathtt{a}_{\dot{\tau}(1)}^{\dag} \dots \mathtt{a}_{\dot{\tau}(n)}^{\dag} \,|0 \rangle = \mathrm{v}_n(C_{\dot{\sigma}}, C_{\tau}).$$
\end{lemma}

\begin{proof}
We have
\begin{align*}
\langle 0|\, \mathtt{a}_{\dot{\sigma}(n)} \dots \mathtt{a}_{\dot{\sigma}(1)} \, \mathtt{a}_{\dot{\tau}(1)}^{\dag} \dots \mathtt{a}_{\dot{\tau}(n)}^{\dag} \,|0 \rangle \ 
& = \sum_{\substack{\upsilon \in \mathfrak{S}_n \\ \dot{\sigma} = \dot{\tau} \circ \upsilon}} \prod_{s \in [n]}\ \prod_{\substack{t \in [n] \setminus [s] \\ \upsilon(s) \, > \, \upsilon(t)}}q_{\dot{\sigma}(s), \dot{\tau} \circ \upsilon(t)}  \\
& = \sum_{\substack{\upsilon \in \mathfrak{S}_n \\ \dot{\sigma} = \dot{\tau} \circ \upsilon}} \prod_{\substack{\{s,t\} \, \in \, \binom{[n]}{2} \\ s \,< \, t \\ \upsilon(s) \, > \, \upsilon(t)}} q_{\dot{\sigma}(s), \dot{\tau} \circ \upsilon(t)}.
\end{align*}
For every $\sigma \in \mathrm{p}^{-1}(\dot{\sigma})$, and $\tau \in \mathrm{p}^{-1}(\dot{\tau})$, we have, on one side, $$\dot{\sigma} = \dot{\tau} \circ \upsilon \, \Longleftrightarrow \, \mathfrak{S}_{\lambda} \sigma = \mathfrak{S}_{\lambda} \tau \upsilon \, \Longleftrightarrow \, \upsilon \in \tau^{-1} \mathfrak{S}_{\lambda} \sigma.$$
On the other side, if $\mathrm{id}$ is the identity permutation, there exists $\varphi \in \mathfrak{S}_{\lambda}$ such that $\upsilon = \tau^{-1} \varphi \sigma$, and
$$\dot{\tau} \circ \upsilon(t) = \dot{\tau} \circ \tau^{-1} \varphi \sigma(t) = \dot{\mathrm{id}} \circ \varphi \sigma(t) = \dot{\mathrm{id}} \circ \sigma(t) = \dot{\sigma}(t).$$
Then,
\begin{align*}
\langle 0|\, \mathtt{a}_{\dot{\sigma}(n)} \dots \mathtt{a}_{\dot{\sigma}(1)} \, \mathtt{a}_{\dot{\tau}(1)}^{\dag} \dots \mathtt{a}_{\dot{\tau}(n)}^{\dag} \,|0 \rangle \ 
& = \sum_{\varphi \in \mathfrak{S}_{\lambda}} \prod_{\substack{\{s,t\} \, \in \, \binom{[n]}{2} \\ s \,< \, t \\ \tau^{-1} \varphi \sigma(s) \, > \, \tau^{-1} \varphi \sigma(t)}} q_{\dot{\sigma}(s), \dot{\sigma}(t)} \\
& = \mathrm{v}_n(C_{\dot{\sigma}}, C_{\tau})
\end{align*}
\end{proof}

\section{The Conjecture of Zagier} \label{SecZ}

\noindent To prove the extended Zagier's conjecture, we first have to return to the general case of hyperplane arrangements. The set $F_{\mathcal{A}}$ is a monoid with product $FG$ defined, for $F, G \in F_{\mathcal{A}}$, by $$\epsilon_H(FG) = \begin{cases}
\epsilon_H(F) & \text{if}\ \epsilon_H(F) \neq 0, \\ \epsilon_H(G) & \text{otherwise}. \end{cases}$$
It is also a meet-semilattice with partial order $\preceq$ defined, for $F, G \in F_{\mathcal{A}}$, by
$$F \preceq G \ \Longleftrightarrow \ \epsilon_H(F) = \epsilon_H(G) \ \, \text{whenever} \ \, \epsilon_H(F) \neq 0.$$ 

\noindent Denote by $O$ the face in $F_{\mathcal{A}}$ such that, for every $H \in \mathcal{A}$, $\epsilon_H(O) = 0$. The \textbf{rank} of a face $F \in F_{\mathcal{A}}$ is $\displaystyle \mathrm{rk}\,F := \dim \bigcap_{\substack{H \in \mathcal{A} \\ F \subseteq H}}H - \dim O$. The \textbf{weight} of a face $F \in F_{\mathcal{A}}$ is the monomial $\displaystyle \mathrm{b}_F := \prod_{\substack{H \in \mathcal{A} \\ F \subseteq H}} h_H^+ h_H^-$. Clearly, $\mathrm{b}_F = 0$ if $F \in C_{\mathcal{A}}$.

\noindent A \textbf{nested face} is a pair $(F,G)$ of faces in $F_{\mathcal{A}}$ such that $F \prec G$. For a nested face $(F,G)$, define the set $F_{\mathcal{A}}^{(F,G)} := \{K \in F_{\mathcal{A}}\ |\ F \preceq K \preceq G\}$, and the polynomial $\displaystyle \Delta_{F,G} := \prod_{K \in F_{\mathcal{A}}^{(F,G)}}(1 - \mathrm{b}_K)$. 

\noindent We prove a variant result of Aguiar and Mahajan \cite[Proposition~8.13]{AgMa}.

\begin{proposition}  \label{PrExZ}
Let $\mathcal{A}$ be a hyperplane arrangement in $\mathbb{R}^n$. Each entry of $\gamma^{-1}: M_{\mathcal{A}} \rightarrow M_{\mathcal{A}}$ is an element of $\displaystyle \bigcup_{C \in C_{\mathcal{A}}} \frac{R_{\mathcal{A}}}{\Delta_{O,C}}$.
\end{proposition}

\begin{proof}
As $\det \gamma$ is a polynomial in $R_{\mathcal{A}}$ with constant term $1$, $\gamma$ is consequently invertible.

\noindent For a chamber $D \in C_{\mathcal{A}}$, and a nested face $(A,D)$, let
$\displaystyle \mathrm{m}(A,D) := \sum_{\substack{C \in C_{\mathcal{A}} \\ AC = D}} \mathrm{v}(D,C)\, C$. We prove by backward induction that $\displaystyle \mathrm{m}(A,D) = \sum_{C \in C_{\mathcal{A}}} x_C \gamma(C)$ with $\displaystyle x_C \in \frac{R_{\mathcal{A}}}{\Delta_{A,D}}$.

\noindent We have $\displaystyle \mathrm{m}(D,D) = \gamma(D)$. The opposite of a face $F \in F_{\mathcal{A}}$ is the face $\tilde{F}$ such that, for every $H \in \mathcal{A}$, $\epsilon_H(\tilde{F}) = -\epsilon_H(F)$. One can read in the proof of \cite[Proposition~8.13]{AgMa} that
$$\displaystyle \mathrm{m}(A,D) - (-1)^{\mathrm{rk}\,D - \mathrm{rk}\,A} \, \mathrm{v}(D,A\tilde{D}) \, \mathrm{m}(A,A\tilde{D}) = \sum_{F \in F_{\mathcal{A}}^{(A,D)} \setminus \{A\}} (-1)^{\mathrm{rk}\,F - \mathrm{rk}\,A +1} \mathrm{m}(F,D).$$
By induction hypothesis, for every $C \in C_{\mathcal{A}}$, there exists $\displaystyle a_C \in \bigcup_{F \in F_{\mathcal{A}}^{(A,D)} \setminus \{A\}} \frac{R_{\mathcal{A}}}{\Delta_{F,D}}$, such that $$\sum_{F \in F_{\mathcal{A}}^{(A,D)} \setminus \{A\}} (-1)^{\mathrm{rk}\,F - \mathrm{rk}\,A +1} \mathrm{m}(F,D) = \sum_{C \in C_{\mathcal{A}}} a_C \gamma(C).$$
Remark that, for every $F \in F_{\mathcal{A}}^{(A,D)}$, we have $\mathrm{b}_F = \mathrm{b}_{A\tilde{F}}$, which means that $$\bigcup_{F \in F_{\mathcal{A}}^{(A,A\tilde{D})} \setminus \{A\}} \frac{R_{\mathcal{A}}}{\Delta_{F, A\tilde{D}}} = \bigcup_{F \in F_{\mathcal{A}}^{(A,D)} \setminus \{A\}} \frac{R_{\mathcal{A}}}{\Delta_{F, D}}.$$
Since $A \leq A\tilde{D}$ and $A(\widetilde{A\tilde{D}}) = D$, replacing $D$ with $A\tilde{D}$, there exists also $\displaystyle e_C \in \bigcup_{F \in F_{\mathcal{A}}^{(A,D)} \setminus \{A\}} \frac{R_{\mathcal{A}}}{\Delta_{F, D}}$, such that, for every $C \in C_{\mathcal{A}}$, $\displaystyle \mathrm{m}(A,A\tilde{D}) - (-1)^{\mathrm{rk}\,A\tilde{D} - \mathrm{rk}\,A} \, \mathrm{v}(A\tilde{D},D) \, \mathrm{m}(A,D) = \sum_{C \in C_{\mathcal{A}}} e_C \gamma(C)$.
Therefore, 
\begin{align*}
\mathrm{m}(A,D) - (-1)^{\mathrm{rk}\,D - \mathrm{rk}\,A} \, \mathrm{v}(D,A\tilde{D}) \, \mathrm{m}(A,A\tilde{D}) & = \sum_{C \in C_{\mathcal{A}}} a_C \gamma(C) \\
\mathrm{m}(A,D) - \mathrm{v}(D,A\tilde{D}) \, \mathrm{v}(A\tilde{D},D) \, \mathrm{m}(A,D) & = \sum_{C \in C_{\mathcal{A}}} \big(a_C + (-1)^{\mathrm{rk}\,D - \mathrm{rk}\,A} \, \mathrm{v}(D,A\tilde{D}) \, e_C \big) \gamma(C) \\
\mathrm{m}(A,D) & = \sum_{C \in C_{\mathcal{A}}} \frac{a_C + (-1)^{\mathrm{rk}\,D - \mathrm{rk}\,A} \, \mathrm{v}(D,A\tilde{D}) \, e_C}{1 - \mathrm{b}_A} \gamma(C),
\end{align*}
with $\displaystyle \, \frac{a_C + (-1)^{\mathrm{rk}\,D - \mathrm{rk}\,A} \, \mathrm{v}(D,A\tilde{D}) \, e_C}{1 - \mathrm{b}_A} \, \in \, \bigcup_{F \in F_{\mathcal{A}}^{(A,D)}} \frac{R_{\mathcal{A}}}{\Delta_{F, D}} \, = \, \frac{R_{\mathcal{A}}}{\Delta_{A, D}}$.

\noindent For every $D \in C_{\mathcal{A}}$, we have $\mathrm{m}(O,D) = D$. Thus $\displaystyle D = \sum_{C \in C_{\mathcal{A}}} x_C \gamma(C)$ with $\displaystyle x_C \in \frac{R_{\mathcal{A}}}{\Delta_{O, D}}$, and $\displaystyle \gamma^{-1}(D) = \sum_{C \in C_{\mathcal{A}}} x_C \, C$. Finally, each entry of $\gamma^{-1}$ is an element of $\displaystyle \bigcup_{C \in C_{\mathcal{A}}} \frac{R_{\mathcal{A}}}{\Delta_{O, C}}$.
\end{proof}

\noindent We can deduce the extended Zagier's conjecture.

\begin{corollary}
Let $n \geq 2$, and suppose that $q_{i,j} = q_{j,i} = q$. Then each entry of $\gamma_n^{-1}$ is an element of $\displaystyle \frac{\mathbb{Z}[q]}{\displaystyle \prod_{i \in [n-1]}\big(1 - q^{i^2 + i}\big)^{n-i}}$. 
\end{corollary}

\begin{proof} Let $\displaystyle O_n = \bigcap_{H \in \mathcal{A}_n} H$, $C_{\sigma} \in C_{\mathcal{A}_n}$, and $i,j \in [n-1]$ such that $i \geq j$. A $i$-dimensional face $F_i \in F_{\mathcal{A}_n}^{(O_n, C_{\sigma})}$ has the form
$$F_i = \{x \in \mathbb{R}^n\ |\ x_{\sigma(1)} < \dots < \overbrace{x_{\sigma(j)} = x_{\sigma(j+1)} = \dots = x_{\sigma(j+n-i+1)}}^{\text{$n-i+1$ variables}} < \dots < x_{\sigma(n)}\}.$$ 
Thus $$\displaystyle \mathrm{b}_{F_i} = \prod_{\{k,l\} \, \in \, \binom{I_i}{2}} q_{k,l} q_{l,k} \quad \text{with} \quad I_i = \big\{\sigma(j), \sigma(j+1), \dots, \sigma(j+n-i+1)\big\}.$$
If $q = q_{k,l} = q_{l,k}$, then $\mathrm{b}_{F_i} = q^{(n-i)^2 + (n-i)}$. Since $F_{\mathcal{A}_n}^{(O_n, C_{\sigma})}$ contains exactly $i$ $i$-dimensional face, we get $\displaystyle \Delta_{O_n, C_{\sigma}} = \prod_{i \in [n-1]}\big(1 - q^{(n-i)^2 + (n-i)}\big)^i$. Finally, we deduce from Proposition \ref{PrExZ} that each entry of $\gamma_n^{-1}$ is an element of $\displaystyle \frac{\mathbb{Z}[q]}{\displaystyle \prod_{i \in [n-1]}\big(1 - q^{(n-i)^2 + (n-i)}\big)^i}$.
\end{proof}

\newpage

\bibliographystyle{abbrvnat}

\end{document}